\documentclass{amsproc}

\usepackage{color}
\usepackage{amsmath}
\usepackage{amsthm}
\usepackage{bm}
\usepackage{multirow}
\usepackage{graphicx}

\DeclareMathOperator{\trace}{tr}

\def\L{\mathbb L}

\newcommand*{\Id}{\operatorname{Id}}

\newtheorem{theorem}{Theorem}[section]

\newtheorem{proposition}[theorem]{Proposition}

\theoremstyle{definition}

\theoremstyle{remark}

\numberwithin{equation}{section}

\begin{document}

\title{Canards in modified equations for Euler discretizations}



\author{Maximilian Engel}
\address{Department of Mathematics and Computer Science, Freie Universit\"at Berlin, Germany}
\email{maximilian.engel@fu-berlin.de}
\thanks{ME was supported by Germany’s Excellence Strategy – The Berlin Mathematics Research Center MATH+ (EXC-2046/1, project ID:
390685689), via projects AA1-8 and AA1-18. Furthermore, ME thanks the DFG SPP 2298 and SFB 1114 (project A08) for support.}

\author{Georg A. Gottwald}
\address{School of Mathematics and Statistics, The University of Sydney, 2006 NSW, Australia}
\email{georg.gottwald@sydney.edu.au}
\thanks{GAG was partially supported by the Australian Research Council grant DP180101385. The authors thank H.~Jard\'on-Kojakhmetov for helpful comments.}


\date{today}

\begin{abstract}
Canards are a well-studied phenomenon in fast-slow ordinary differential equations implying the delayed loss of stability after the slow passage through a singularity. Recent studies have shown that the corresponding maps stemming from explicit Runge-Kutta discretizations, in particular the forward Euler scheme, exhibit significant distinctions to the continuous-time behavior: for folds, the delay in loss of stability is typically shortened whereas, for transcritical singularities, it is arbitrarily prolonged. We employ the method of modified equations, which correspond with the fixed discretization schemes up to higher order, to understand and quantify these effects directly from a fast-slow ODE, yielding consistent results with the discrete-time behavior and opening a new perspective on the wide range of (de-)stabilization phenomena along canards.
\end{abstract}

\maketitle
\section{Introduction}
\label{sec:Intro}
Dynamical systems on multiple time scales are extensively studied and exhibit rich behavior, stemming from certain types of ODEs or PDEs that are often models of real-world systems in biology, physics or the social sciences. 
Given the relevance for applications there is high interest in the features of numerical discretizations of such models. The numerical discretization of multiscale systems is generally challenging.
 For stability reasons, the time step typically needs to be adjusted to capture the fast dynamics. This implies that in order to capture the more relevant slow dynamics one faces high computational costs. 
 Furthermore, critical changes, for example due to bifurcations, may occur, causing complicated behavior that is difficult to capture via numerical schemes. 
 For these reasons but also as a starting point for understanding intricate fast-slow phenomena in discrete time, there has been growing interest in recent years in studying the multiscale behavior of discretization schemes in the context of singularities requiring specific geometric methods (see e.g.~\cite{EngelJardonSIADS, EngelKuehn19, EngeletalJNLS, JelbartKuehn22, NippStoffer2013}).

The standard problem is to consider a system of singularly perturbed ordinary differential equations (ODEs) on the \textit{slow} time scale
\begin{equation}\label{eq:gen_slowequ}
    \begin{split}
    \varepsilon \dot x &= f(x,y,\epsilon)\,, \\
    \dot y &= g(x,y,\epsilon)\,,  \quad \ x \in \mathbb{R}^m, 
\quad y \in \mathbb R^n, \quad 0 < \epsilon \ll 1\,,
    \end{split}
\end{equation}
with \textit{critical manifold} 
\begin{equation}
    S_0= \{(x,y) \in \mathbb{R}^{m+n} \,:\, f(x,y,0) = 0 \}\,,
\end{equation}
where $\epsilon\ll 1$ quantifies the degree of the time scale separation. 
We call the set $S_0$ \textit{normally hyperbolic} if, for all $p\in S_0$,
the Jacobian $\textnormal{D}_xf(p)\in\mathbb{R}^{m\times m}$ has no
eigenvalue on the imaginary axis. 
The by now classical \textit{Fenichel 
Theory} \cite{Fenichel4,Jones,KuehnBook} says that, if $S_0$ is normally hyperbolic and compact, then there is a locally invariant \emph{slow manifold} $S_{\varepsilon}$,
behaving like a regular perturbation of $S_0$, 
for all $\varepsilon$ 
sufficiently small. 
On the other hand, \emph{loss of normal hyperbolicity}, which occurs whenever $\textnormal{D}_xf(p)$ has at least one eigenvalue on the imaginary axis, is known to be responsible for many complicated dynamic effects, such as \emph{canards}. 

Here, we focus on planar fast-slow systems with a \emph{canard point} at the origin, past whom trajectories connect an attracting branch of the slow manifold with a repelling one, also described as \emph{maximal canard}~\cite{BenoitCallotDienerDiener,DumortierRoussarie96}. 
For continuous-time fast-slow systems of the form~\eqref{eq:gen_slowequ}, such canard solutions characterize the delay in the onset of instabilities when trajectories slowly cross a singularity and continue for some time near the unstable part of the invariant manifold \cite{DeMaes2008, de2016entry, fruchard2009survey, hayes2016geometric}.
Two important types of such canards occur in fold \cite{KrupaSzmolyan01a} and in transcritical singularities \cite{KrupaSzmolyan01b}. Their respective Euler discretizations constitute simple fast-slow maps with interesting phenomena, and exhibit different effects to those of the continuous time ODEs they were designed to model. While the fold problem contains conserved quantities up to first order which are not captured by simple explicit forward schemes and seems to be inaccessible via the Euler method \cite{EngelJardonSIADS, EngeletalJNLS}, the transcritical Euler map still exhibits canards but with the intriguing effect of discretization-induced stabilization; in other words, one can observe the extended loss of stability compared to the corresponding canonical ODE \cite{EngelJardonSIADS, EngelKuehn19}.

The crucial idea of this article is to understand these effects better via the tool of \emph{modified equations}. Rather than studying the discrete dynamical system provided by the Euler discretization as done in \cite{EngelJardonSIADS, EngeletalJNLS}, the concept of backward error analysis allows for the description of the behaviour of the discrete system by a continuous dynamical system, the modified equation, which is asymptotically close to the original dynamical system in the fixed time step $h$ of the Euler discretization. This has the advantage of having at our disposal the rich gamut of analytical tools available for continuous-time systems. Backward error analysis has been widely used in numerical analysis to study the stability of numerical schemes, see for example the excellent textbooks \cite{HairerLubichWanner,LeimkuhlerReich}.

Our two main findings are that, by virtue of the modified equations, we can (a) find parameter regimes to stabilize canard phenomena for the Euler map, formulating the modified equation for the Euler discretization within the corresponding normal form framework of folded canard points (see Proposition~\ref{prop:fold}), and (b) capture the stabilization effect for the transcritical problem in terms of way-in/way-out relations for the respective modified equation in a straightforward manner (see Proposition~\ref{prop:psi_transcritical}).
We present our observations according to the following structure: Section~\ref{sec:BE} introduces the key ideas of backward error analysis and derives the general form of a second order modified equation for Euler schemes.
In Section~\ref{sec:fold}, we discuss properties of the modified equation for the Euler discretization of the most simple canonical form of a fold singularity with canards, studying its stability and way-in/way-out map in Section~\ref{sec:fold_mod_Euler}. We then show that this modified equation, in fact, fits the general normal form of such folded canards upon addition of a parameter $\lambda$, yielding a prediction of maximal canards and Hopf bifurcation in terms of $\lambda$, depending on the time separation parameter $\varepsilon$; this allows for obtaining parameter regimes under which the Euler map approximates typical canard behavior more accurately.
Section~\ref{sec:transcritical} discusses the analogous analysis to Section~\ref{sec:fold} for the canonical form of transcritical canards, giving exactly the same prediction of their arbitrarily long stabilization as the discrete time problem. Hence, we establish a continuous-time example with full stabilization of canards along the repelling critical branch.
We conclude with a discussion and outlook in Section~\ref{sec:discuss}.


\section{Backward error analysis and the modified equations}
\label{sec:BE}
Backward error analysis \cite{HairerLubichWanner,LeimkuhlerReich} allows for the study of finite time step effects of discrete numerical schemes by studying continuous time dynamics of so called modified equations. The main idea is that a numerical scheme approximates a modified equation to a higher accuracy in the time step $h$ than the actual ODE one set out to solve. Consider the ODE
\begin{align}
\dot z = f_0(z).
\label{eq:ODE}
\end{align}
Its Euler discretization is
\begin{align}
z_{n+1} = z_n  + h f_0(z_n),
\label{eq:Euler}
\end{align}
where $h$ denotes the discrete time step. The Euler discretization is an $\mathcal{O}(h)$ discretization of the ODE. We can construct a modified ODE
\begin{align}
\dot{\tilde z} = f_h(\tilde z)
\end{align} 
for which the Euler discretization is, for example, an $\mathcal{O}(h^2)$ approximation. Hence, solutions of the numerical scheme better represent solutions of this modified equation than those of the original ODE. To construct the modified equation we expand the modified vectorfield $f_h(\tilde z)$ in powers of the time step $h$ as 
\begin{align}
f_h(\tilde z) = f_0(\tilde z) + h f_1(\tilde z) + h^2 f_2(\tilde z) +  \mathcal{O}(h^3). 
\end{align}
Taylor expanding the solution $\tilde z$ around  $\tilde z_n$ yields
\begin{align}
\tilde z_{n+1} 
&= \tilde z_n + h f_h(\tilde z_n) + \frac{h^2}{2} \textnormal{D} f_h(\tilde z_n)f_h(\tilde z_n) + \mathcal{O}(h^3) \nonumber  \\
&= \tilde z_n + h f_0(\tilde z_n) + h^2 \left(\frac{1}{2} \textnormal{D}f_0 f_0(\tilde z_n) + f_1(\tilde z_n) \right) +  \mathcal{O}(h^3),
\label{eq:Taylor}
\end{align}
where $\textnormal{D}f$ denotes the Jacobian of $f$.
Hence, taking $f_1(z) = -\tfrac12 \textnormal{D}f_0 f_0(z)$, we obtain a numerical discretization which is now second order, and the Euler scheme (\ref{eq:Euler}), which solves the original ODE (\ref{eq:ODE}) only to first order in $h$, solves the modified equation
\begin{align}
\dot z = f_0(z) - \frac{h}{2}\textnormal{D}f_0 f_0(z) 
\label{eq:ODEmod}
\end{align}
to second order in $h$. Note that the time step $h$ here is finite but fixed.
 

\section{Canard in a fold}
\label{sec:fold}
We consider the simplest form of a fold singularity, admitting a canard connection in a slow-fast system
\begin{align}
\varepsilon \dot x &= -y + x^2
\nonumber
\\
\dot y &= x,
\label{eq:fold}
\end{align}
where $\varepsilon \ll 1$ is a small parameter quantifying the degree of scale separation between the slow variable $y$ and the fast variable $x$. 
The maximal canard solution $(x^*(t), y^*(t)) = (\frac{t}{2}, \frac{t^2}{4} - \frac{\varepsilon}{2})$ of~\eqref{eq:fold} lives on the invariant slow manifold $S_{\varepsilon} = \{y=x^2 - \varepsilon/2 \}$, connecting the stable, attracting branch $S_{\varepsilon}^{\textnormal{a}} = \{(x,y) \in S_{\varepsilon}  \,:\, x <0 \}$ and the unstable, repelling branch $S_{\varepsilon}^{\textnormal{r}} = \{(x,y) \in S_{\varepsilon}  \,:\, x >0 \}$. 
The maximal canard serves as a benchmark for the remarkable phenomenon of trajectories staying close to repelling invariant sets for long times; see e.g.~\cite{EngelJardonSIADS, EngeletalJNLS, KrupaSzmolyan01a, KrupaSzmolyanJDE, Wechselberger2012}.

\subsection{Modified equations for Euler}
\label{sec:fold_mod_Euler}

The Euler discretization of (\ref{eq:fold}) is
\begin{align}
x_{n+1} &= x_n  - \frac{h}{\varepsilon}\left( y_n - x_n^2\right)\nonumber \\
y_{n+1} &= y_n + h x_n,
\label{eq:foldEuler}
\end{align}
where $h$ denotes the discrete time step. To resolve the fast dynamics we need to require $h< \varepsilon$.
It is well-known that the Euler discretization is a deficient method for folded canard phenomena, in particular due to the non-preservation of an integral of motion $H(x,y) = e^{\frac{-2y}{\varepsilon}} \left( y - x^2 + \frac{\varepsilon}{2}\right)$ for (\ref{eq:fold}) (see e.g.~\cite{EngeletalJNLS}), and thereby entailing earlier escape from the vicinity of $S_{\varepsilon}^{\textnormal{r}}$. We investigate the implications for the associated modified equation:

Setting $z=(x,y)$, the modified equation (\ref{eq:ODEmod}) for the Euler discretization (\ref{eq:foldEuler}) becomes
\begin{align}
\varepsilon \dot x &= -y + x^2 + \frac{h}{\varepsilon} x \left(y-x^2 + \frac{\varepsilon}{2}  \right)
\nonumber
\\
\dot y &= x + \frac{h}{2\varepsilon}\left( y-x^2 \right).
\label{eq:foldmod}
\end{align}
The dynamics evolves in the modified equation (\ref{eq:foldmod}) up to $\mathcal{O}(\varepsilon^2)$ on the same manifold $S_{\varepsilon} = \{y=x^2 - \varepsilon/2 \}$ on which the dynamics of the full fold equations (\ref{eq:fold}) evolves.   
Linearizing the modified equation (\ref{eq:foldmod}) around the approximate manifold $S_{\varepsilon}$ yields the eigenvalues
\begin{align}
\lambda_{1,2} &= \frac{1}{\varepsilon^2}
\left[ 
-h x^2 +\frac{h}{4}\varepsilon+\varepsilon x 
\pm 
\sqrt{
(hx^2-\frac{h}{4}\varepsilon - \varepsilon x)^2
-\varepsilon^2(\varepsilon-h x)
}
\right].
\label{eq:lambda}
\end{align}
The eigenvalues are shown for $\varepsilon=0.1$ and $h=10^{-4}$ as a function of $x$ in Figure~\ref{fig:fold_lambda}. Notably, there is a range of initial conditions $x\in[x_1,x_2]$ such that $\lambda$ forms a complex conjugate pair. The boundary of this set is given as solutions of 
\begin{align}
0 = (\frac{h}{\varepsilon}x^2-\frac{h}{4} - x)^2 - \varepsilon + h x.
\label{eq:lambda2}
\end{align}
A Taylor expansion around $h=0$ yields 
\begin{align}
x_{1,2} = \pm \sqrt{\varepsilon} + \frac{h}{4} + \mathcal{O}(h^2).
\label{eq:lambda3}
\end{align}
Note that the range is symmetric around the origin for the original fold equation with $h=0$ but is asymmetric for the modified equation with $h> 0$. We will explain this shift in the subsequent section below.

\medskip

To determine when the dynamics leaves the manifold $S_{\varepsilon}$ we compute the way-in/way-out map $\Psi(\tau)$ along the solution $x=x_0+t/2$ defined as
\begin{align}
\Psi(t) = \int_{t_0}^t \mathcal{R}[\lambda_1](s) ds,
\label{eq:wayinout}
\end{align}
where $\lambda_1$ is the eigenvalue with maximal modulus (cf.~ (\ref{eq:lambda})). Here $t_0$ denotes the time when the solution has approached the manifold $S_{\varepsilon}= \{y=x^2-\varepsilon/2 \}$ at the branch $S_{\varepsilon}^{\textnormal{a}}$. 
Without loss of generality, we set $t_0=0$. When $\Psi(t)=0$, the solution has experienced as much expansion on the unstable branch with $x>0$ as it has experienced contraction along $x<0$. Hence, if $\Psi(\tau)=0$ and $\Psi(t) > 0$ for $t \in (\tau, \tau +\delta)$ for some $\delta > 0$, the solution will exit from the branch $S_{\varepsilon}^{\textnormal{r}} \subset S_{\varepsilon}$. 


In Figure~\ref{fig:fold}, we show the results of numerical simulations. The solution of the original system (\ref{eq:fold}) still follows the manifold $S_{\varepsilon}$ when its Euler discretization already leaves the manifold. It is clearly seen that the modified equation describes the solution of the Euler discretization very well and the exit point is well described by the way-in/way-out map $\Psi(\tau)$ and its roots. 
To find the solutions of the ODEs we employ the Matlab routine ode45 with a pre-set absolute and relative tolerance of $10^{-12}$. We further increase the floating point precision of the Euler discretization to 50 digits.


\subsection{Canard extensions and Hopf bifurcation}
\label{sec:fold_extensions}
In the normal form of folded canards one can observe singular Hopf bifurcations \cite[Theorem 8.2.1]{KuehnBook} and extensions of maximal canards. To observe those, one has to add a parameter $\lambda$ which in the simplest form yields
\begin{align}
\varepsilon \dot x &= -y + x^2
\nonumber
\\
\dot y &= x - \lambda,
\label{eq:fold_lambda}
\end{align}
 Observe that for $\lambda=0$, this is the same as~\eqref{eq:fold}. The equilibrium for the normal form \eqref{eq:fold_lambda} is at $(x,y)=(\lambda, \lambda^2)$ with Jacobian
 $$ J(x=\lambda,y=\lambda^2)=\frac{1}{\epsilon} \begin{pmatrix}  2 \lambda & -1 \\ \epsilon & 0
 \end{pmatrix}.$$
Hence, there is a Hopf bifurcation at $\lambda =0$ with equilibrium $(0,0)$ which is, in fact, degenerate, and only becomes non-degenerate if additional terms are included in \eqref{eq:fold_lambda} to make the first Lyapunov coefficient non-zero.

The modified equation for an Euler discretization  of~\eqref{eq:fold_lambda} reads
\begin{align}
\varepsilon \dot x &= -y + x^2 + \frac{h}{\varepsilon} x \left(y-x^2 \right) + \frac{h}{2} (x-\lambda)
\nonumber
\\
\dot y &= x - \lambda + \frac{h}{2\varepsilon}\left( y-x^2 \right).
\label{eq:foldmod_lambda}
\end{align}
Recall that we require $h < \varepsilon$.
The Jacobian at any point $(x,y)$ is given by
 $$ J(x,y)=\frac{1}{\epsilon} \begin{pmatrix}  2x + \frac{h}{\epsilon}y - 3 \frac{h}{\epsilon} x^2 + \frac{h}{2} & -1 + \frac{h}{\epsilon}x \\ \epsilon - hx & \frac{ h}{2}
 \end{pmatrix}.$$
 Hence, at the equilibrium $x= y =\lambda=0$, we now have that $\trace[J(0,0)] = \frac{h}{\epsilon}$. This suggests that there may be a bifurcation for some $\lambda$-dependent equilibrium for $\lambda <0$.  
 Indeed, we can check this by using results from \cite{KrupaSzmolyanJDE}. 
 %
%
In more detail, we write equation~\eqref{eq:fold_lambda} in the fast time scale $\tau = t/\epsilon$:
\begin{align}
x' &= -y + x^2
\nonumber
\\
 y' &= \varepsilon(x - \lambda).
\label{eq:fold_lambda_fast}
\end{align}
The associated Euler discretization is 
\begin{align}
x_{n+1} &= x_n  - \tilde h \left( y_n - x_n^2\right)\nonumber \\
y_{n+1} &= y_n + \tilde h \varepsilon (x_n - \lambda),
\end{align}
with $\tilde h = h/\varepsilon$, and its corresponding modified equation reads 
\begin{align}
 x' &= -y + x^2 + \tilde h x \left(y-x^2 \right) + \frac{\tilde h \varepsilon}{2} (x-\lambda)
\nonumber
\\
y' &= \varepsilon(x - \lambda) + \frac{\varepsilon \tilde h}{2}\left( y-x^2 \right).
\label{eq:foldmod_lambda_fast}
\end{align}
Note that $\tilde p = (\lambda, \lambda^2)$, which is the origin for $\lambda=0$, remains an equilibrium for the modified equation. The modified equations~\eqref{eq:foldmod_lambda} and~\eqref{eq:foldmod_lambda_fast} coincide under a combination of the continuous time change $\tau = t/\varepsilon$ and the discrete time change $\tilde h = h/\varepsilon$.

We obtain the following observation:
\begin{proposition}
\label{prop:fold}
For sufficiently small $\varepsilon$ and $0< h < \varepsilon$,  system~\eqref{eq:foldmod_lambda_fast} exhibits a Hopf bifurcation at 
\begin{equation}
    \lambda_{\text{H}}(\sqrt{\epsilon}) = - \frac{h}{2} + \mathcal O \left( h \sqrt{\epsilon} \right),
    \label{eq:lambdaHopf}
\end{equation}
and the existence of a maximal canard at the same value of $\lambda$ up to first order, i.e.~at
\begin{equation}
     \lambda_{\text{C}}(\sqrt{\epsilon}) = - \frac{h}{2}  + \mathcal O \left( h \sqrt{\epsilon}\right).
    \label{eq:lambdaCanard}
\end{equation}
\end{proposition}
\begin{proof}
We make direct use of the fact that equation~\eqref{eq:foldmod_lambda_fast} fits the canonical form of a non-degenerate canard point \cite[Section 3.2]{KrupaSzmolyanJDE} with factors
\begin{align*}
    h_1(x,y,\lambda, \varepsilon) &= 1-\tilde h x, \quad h_2(x,y,\lambda, \varepsilon) = 1- \tilde h x, \\
    \quad h_3(x,y,\lambda, \varepsilon) &= \frac{\tilde h}{2} (x-\lambda), \quad
     h_4(x,y,\lambda, \varepsilon) = 1-\frac{\tilde h}{2}x, \\
     h_5(x,y,\lambda, \varepsilon) &= 1, \quad h_6(x,y,\lambda, \varepsilon) = \frac{\tilde h}{2}.
\end{align*}
This implies that the constants determining Hopf bifurcations and canard extensions are $a_1 = \frac{\tilde h}{2}, a_2 = - \tilde h, a_3 = -\tilde h, a_4 = - \frac{\tilde h}{2}, a_5 = \frac{\tilde h}{2}$ and $A=0$.
Hence, upon resubstituting $\tilde h=h/\varepsilon$, we can apply \cite[Theorem 3.1]{KrupaSzmolyanJDE} to deduce that, for sufficiently small $\epsilon > 0$, there is a (potentially degenerate) Hopf bifurcation at
\begin{equation*}
    \lambda_{\text{H}}(\sqrt{\epsilon}) = - \frac{h}{2} + \mathcal O \left( h \sqrt{\epsilon} \right),
\end{equation*}
i.e.~$\tilde p$ is stable for $\lambda < \lambda_{\text{H}}(\sqrt{\epsilon})$ and loses stability through a Hopf bifurcation as $\lambda$ passes through $\lambda_{\text{H}}(\sqrt{\epsilon})$. Additionally, we have, by \cite[Theorem 3.1]{KrupaSzmolyanJDE}, the existence of a maximal canard at the same value of $\lambda$ up to first order, i.e.~at
\begin{equation*}
     \lambda_{\text{C}}(\sqrt{\epsilon}) = - \frac{h}{2}  + \mathcal O \left( h \sqrt{\epsilon}\right).
\end{equation*}
This shows the claims.
\end{proof}
Hence, one may say that the modified equation preserves the canard phenomenon upon variation of the additional parameter $\lambda$. The implications for the Euler scheme are illustrated in Figure~\ref{fig:fold_Hopf} where for $\lambda=-h/2$ the dynamics is indicative of a periodic solution, whereas for $\lambda=0$ the solution shows the same escape beahviour as in Figure~\ref{fig:fold}.


Indeed, this value coincides (up to leading orders in $h$ and $\varepsilon$) with the value at which we observe change of stability along the curve $S_{\varepsilon} = \{y =x^2 - \varepsilon/2\}$ for solutions of equation~\eqref{eq:foldmod} or~\eqref{eq:foldmod_lambda} with $\lambda=0$. In more detail, we observe that $\mathcal R \left[ \lambda_{1,2} (x^*)\right] = 0$ for $x^* = \frac{\epsilon - \sqrt{\epsilon^2 + h^2 \epsilon}}{2h}$, cf.~equation~\eqref{eq:lambda}. A Taylor expansion in $h$ shows readily that $\sqrt{1 + 2\frac{h^2}{\varepsilon}} = 1 + \frac{h^2}{\epsilon} + \mathcal{O}(h^3)$, and, hence,
\begin{align*}
    x^*(h, \varepsilon) &= \frac{\varepsilon}{2h}\left( 1- \sqrt{1 + 2 \frac{h^2}{\varepsilon}} \right)  \\
    &= \frac{\varepsilon}{2h}\left( - \frac{ h^2}{\varepsilon} + \mathcal{O}(h^3)\right) =  - \frac{h}{2} + \mathcal{O}(h^2).
\end{align*}  

\begin{figure}
	\centering
	\includegraphics[width=0.8\linewidth]{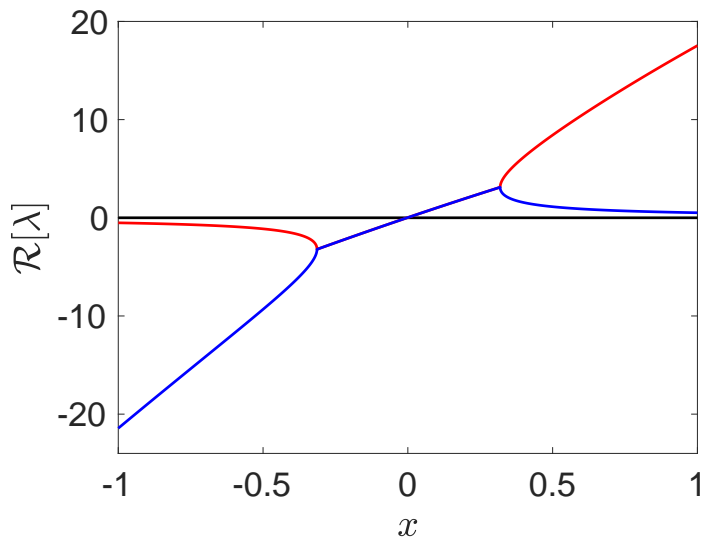}
	\caption{Real part of the two eigenvalues $\lambda_{1,2}$ (\ref{eq:lambda}) for the fold equation (\ref{eq:fold}) with $\varepsilon=0.1$ and time step $h=0.01$. In the range $x\in[-0.314,0.319]
$ the eigenvalues are a complex conjugate pair.}
	\label{fig:fold_lambda}
\end{figure}

\begin{figure}
	\centering
	\includegraphics[width=0.8\linewidth]{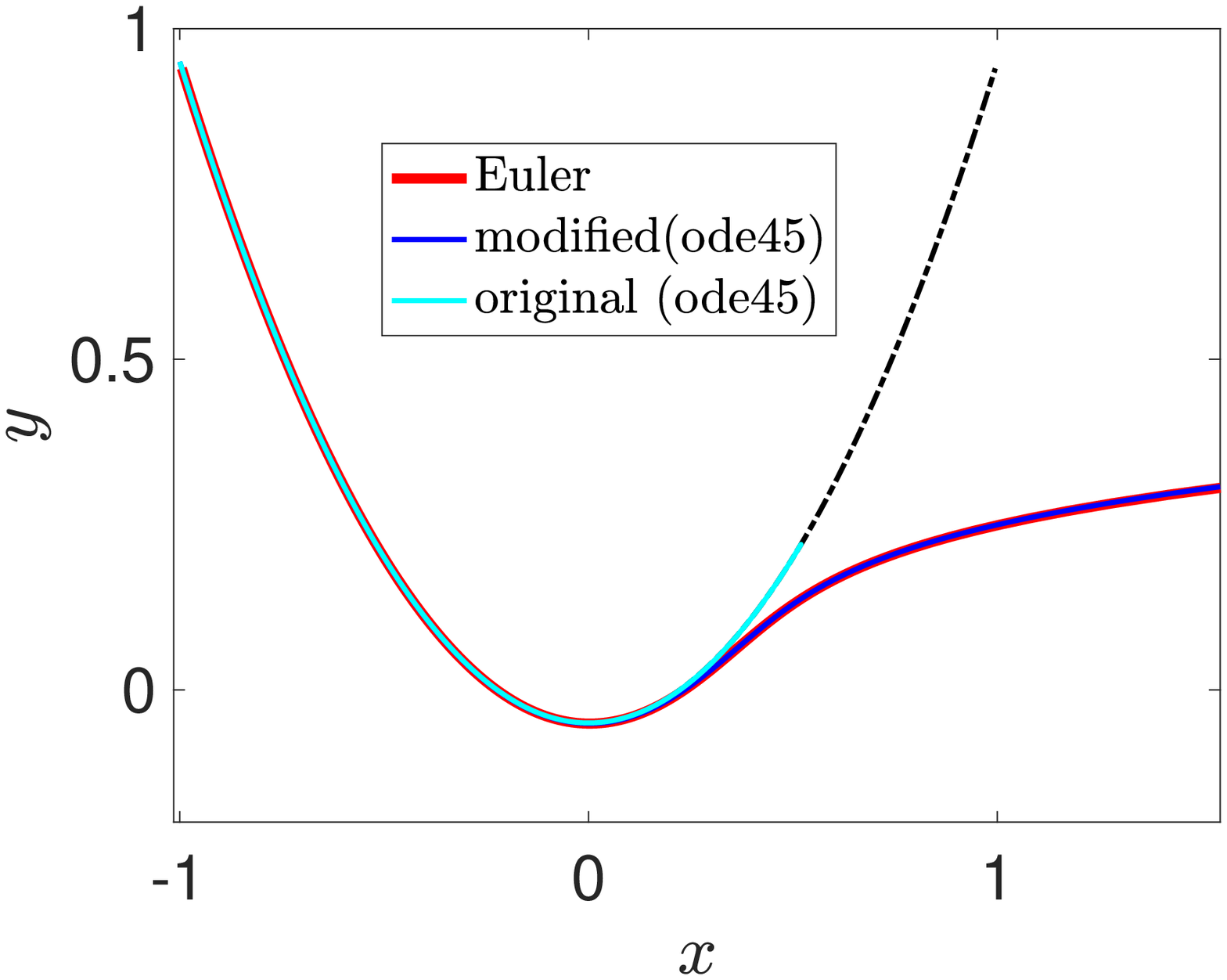}\\
	\includegraphics[width=0.8\linewidth]{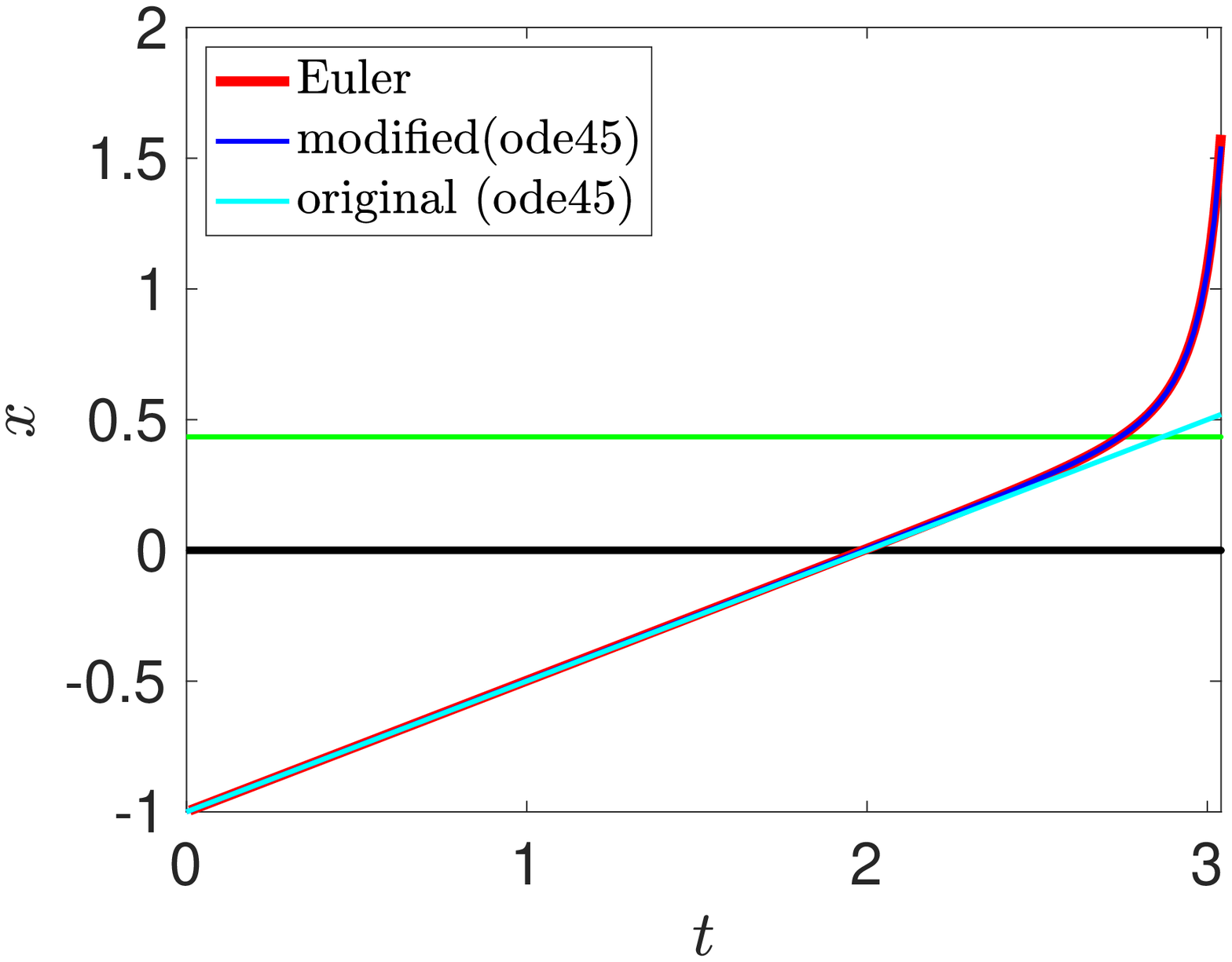}\\
	\caption{Numerical simulations for the fold equation (\ref{eq:fold}) with $\varepsilon=0.1$. We employ a time step of $h=0.01$ for the Euler discretization. Top: Dynamics of the Euler discretization (online red), a high-order simulation of the original fold equation (online cyan) and a high-order simulation of the modified equation (\ref{eq:foldmod}).  The dashed line shows the approximate manifold $y=x^2 - \varepsilon/2$. Bottom: Plot of $x$ as a function of time. The horizontal line shows the value $x = x_0 + \tau/2$ where $\tau$ is the time for which the way-in/way-out map $\Psi(\tau)=0$.  
 }
	\label{fig:fold}
\end{figure}

\begin{figure}
	\centering
	\includegraphics[width=0.8\linewidth]{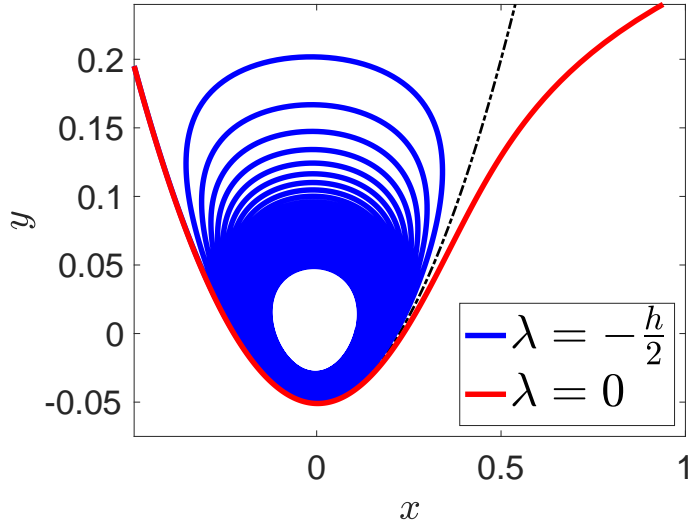}
	\caption{Numerical simulations for the fold equation (\ref{eq:fold_lambda}) with $\varepsilon=0.1$. We employ a time step of $h=0.01$ for the Euler discretization. Shown are results for $\lambda=0$ (cf. Figure~\ref{fig:fold}a (online red) and for the critical parameter $\lambda=\lambda_H = -h/2$ (online blue) which indicates relaxation onto a limit cycle. The dashed line shows the approximate manifold $S_\varepsilon$ with $y=x^2 - \varepsilon/2$.}
	\label{fig:fold_Hopf}
\end{figure}



\section{Canard in a transcritical singularity}
\label{sec:transcritical}
As a second study, we now consider the simplest canonical form of a slow-fast system with a transcritical canard singularity 
\begin{align}
\dot x &= x^2 - y^2 + \varepsilon
\nonumber
\\
\dot y &= \varepsilon,
\label{eq:trans}
\end{align}
where $\varepsilon \ll 1$ again quantifies the degree of scale separation between the slow variable $y$ and the fast variable $x$. The associated Euler discretization with time step $h$ reads
\begin{align}
\label{eq:transEul}
x_{n+1} &= x_n  + h\left( x_n^2-y_n^2 + \varepsilon \right) \nonumber \\
y_{n+1} &= y_n + h \varepsilon.
\end{align}
The associated modified equation can be readily evaluated as
\begin{align}
\dot x &= (1-hx) (x^2 - y^2 + \varepsilon) + \varepsilon h x
\nonumber
\\
\dot y &= \varepsilon.
\label{eq:transmod}
\end{align}
Note that in the two continuous time systems (\ref{eq:trans}) and (\ref{eq:transmod}) as well as in the discrete Euler system (\ref{eq:transEul}) the line $S= \{(x,y) \in \mathbb R^2 \,:\, x=y\}$ is invariant.
For equations~\eqref{eq:trans} and~\eqref{eq:transEul} we have the stable and unstable branches
 $S^{\textnormal{a}} = \{ (x,y) \in S \,:\, x < 0\}$ and $S^{\textnormal{r}} = \{ (x,y) \in S \,:\, x > 0\}$.
 For equation~\eqref{eq:transmod}, the non-zero eigenvalue of the Jacobian when linearized around the solution $y=x$ is $\lambda = 2x(1-hx)$. 
To determine when the dynamics leaves the unstable branch $S^{\textnormal{r}}$, we compute again the way-in/way-out map $\Psi(t)$ (cf. (\ref{eq:wayinout})) along the solution $y=x=x_0+\varepsilon t$ for equation~\eqref{eq:transmod}.  As for the fold equation, when $\Psi(\tau)=0$, the solution has experienced as much expansion on the unstable branch with $S^{\textnormal{r}}$ as it has experienced contraction along $S^{\textnormal{a}}$.
\begin{proposition}
\label{prop:psi_transcritical}
     The way-in/way-out map $\Psi(t)$ along the canard solution $y=x=x_0+\varepsilon t$ for the modified equation~\eqref{eq:transmod} takes the form
     \begin{align} \label{eq:Psi}
\Psi(t) = 2x_0(1- h x_0) t + \varepsilon t^2 - 2 \varepsilon h x_0 t^2 - \frac{2}{3} \varepsilon^2 h t^3,
\end{align}
yielding the following cases depending on the initial condition $x_0$:
\begin{itemize}
    \item $ - \frac{1}{2h} < x_0 < 0$: there are $t_1, t_2 > 0$ such that $\Psi(t_1) = \Psi(t_2) = 0$ and $\Psi(t) > 0$ for all $t \in (t_1, t_2)$.
    \item $x_0 = - \frac{1}{2h}$: $\Psi(t^*) = 0$ for $t^*= \frac{3}{2h\varepsilon}$ and $\Psi (t) < 0$ for all other $t>0$.
    \item $x_0 < - \frac{1}{2h}$: $\Psi (t) < 0$ for all $t>0$.
\end{itemize}
\end{proposition}
\begin{proof}
Substituting the eigenvalue $\lambda = 2x(1-hx)$ into the definition of the way-in/way-put map (\ref{eq:wayinout}) and subsequent integration yields formula~\eqref{eq:Psi}.
The second claim follows by a simple analysis of the parabola 
$$ f(t) = 2x_0(1- h x_0)  + (\varepsilon - 2 \varepsilon h x_0) t - \frac{2}{3} \varepsilon^2 h t^2,$$
where $\Psi(t) = t f(t)$.
\end{proof}
The key insight of Proposition~\ref{prop:psi_transcritical} is that our results based on the continuous-time modified equation are consistent with the discrete-time analysis in~\cite{EngelJardonSIADS}.
 There it is shown that for $x_0 \downarrow - \frac{1}{2h}$ the time $\tau > 0$ such that $\Psi(\tau)=0$ becomes arbitrarily large (cf.~Figure~\ref{fig:trans_x0}). 
  And here, in equation~\eqref{eq:Psi}, one obtains an analogous behavior:  for $x_0 > - \frac{1}{2h}$, the solution will exit from $S$ just after $t=\tau$ with $\Psi(\tau)=0$. Figure~\ref{fig:trans} shows results from a numerical simulation confirming that the modified equation is able to determine the point at which the Euler dynamics leaves the unstable branch for this case.
For $x_0 = - \frac{1}{2h}$ , we have $\Psi (t) < 0$ for all $t>0$ apart from $t^*= \frac{3}{2h\varepsilon}$, where $\Psi(t^*) = 0$. Hence, at this time contraction and expansion have compensated but immediately afterwards contraction takes over again. In other words, up to linear approximation, there is no escape for any $\tau > 0$.
In particular, our findings imply that equation~\eqref{eq:transmod} exhibits canards with arbitrarily long stabilization, establishing a continuous-time example of that behavior. 

For the numerical simulations depicted in Figures~\ref{fig:trans} and~\ref{fig:trans_x0}, we use the Matlab routine ode45 with a pre-set absolute and relative tolerance of $10^{-12}$ (as before for the fold case). We further increase the floating point precision of the Euler discretization
to 100 digits. This mitigates the possibility that the observed stabilization of the canard for $x_0=-1/2h$, as seen in Figure~\ref{fig:trans_x0}, is due to numerical round of errors (cf.~also \cite{EngelJardonSIADS}).


\begin{figure}
	\centering
	\includegraphics[width=0.65\linewidth]{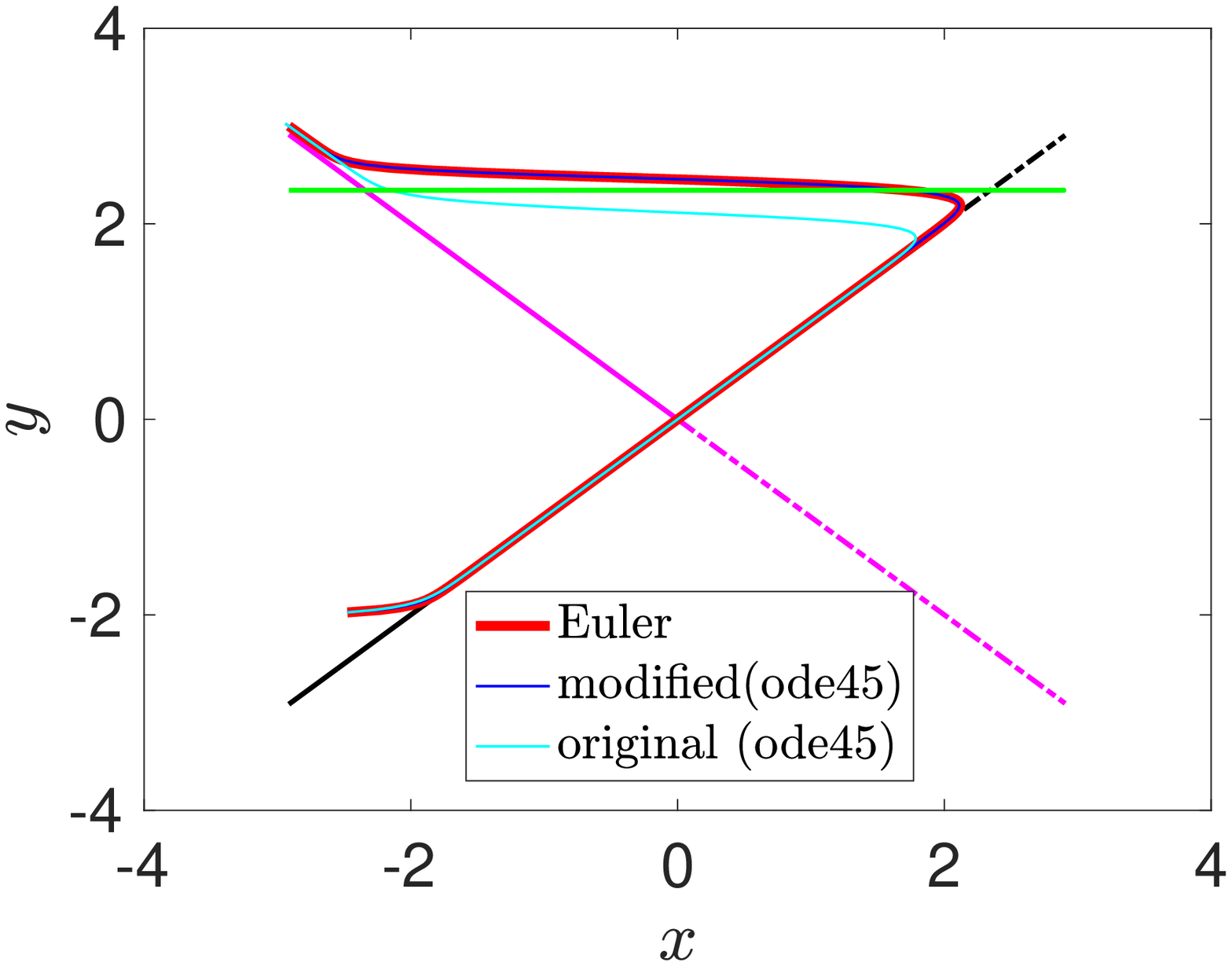}\\
	\includegraphics[width=0.65\linewidth]{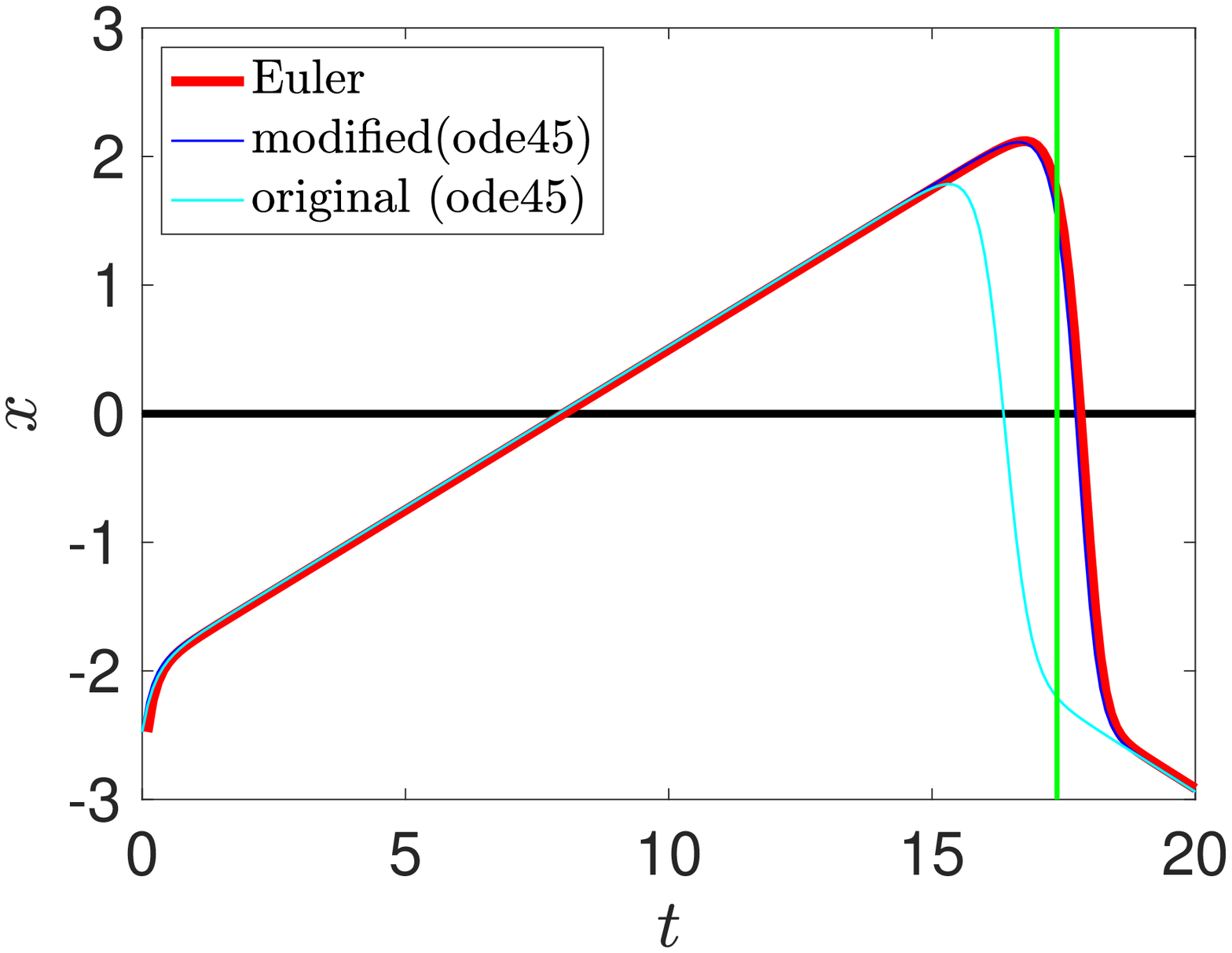}
	\caption{Numerical simulations for the transcritical equation (\ref{eq:trans}) with $\varepsilon=0.25$. We employ a time step of $h=0.1$ for the Euler discretization. Top: Dynamics of the Euler discretization (online red), a high-order simulation of the original transcritical equation (online cyan) and a high-order simulation of the modified equation (\ref{eq:transmod}). The dashed lines show the critical manifold $y=\pm x$. Bottom: Plot of $x$ as a function of time. The vertical line shows the time $t=\tau$ for which the way-in/way-out map $\Psi(\tau)=0$.}
	\label{fig:trans}
\end{figure}

\begin{figure}
	\centering
	\includegraphics[width=0.7\linewidth]{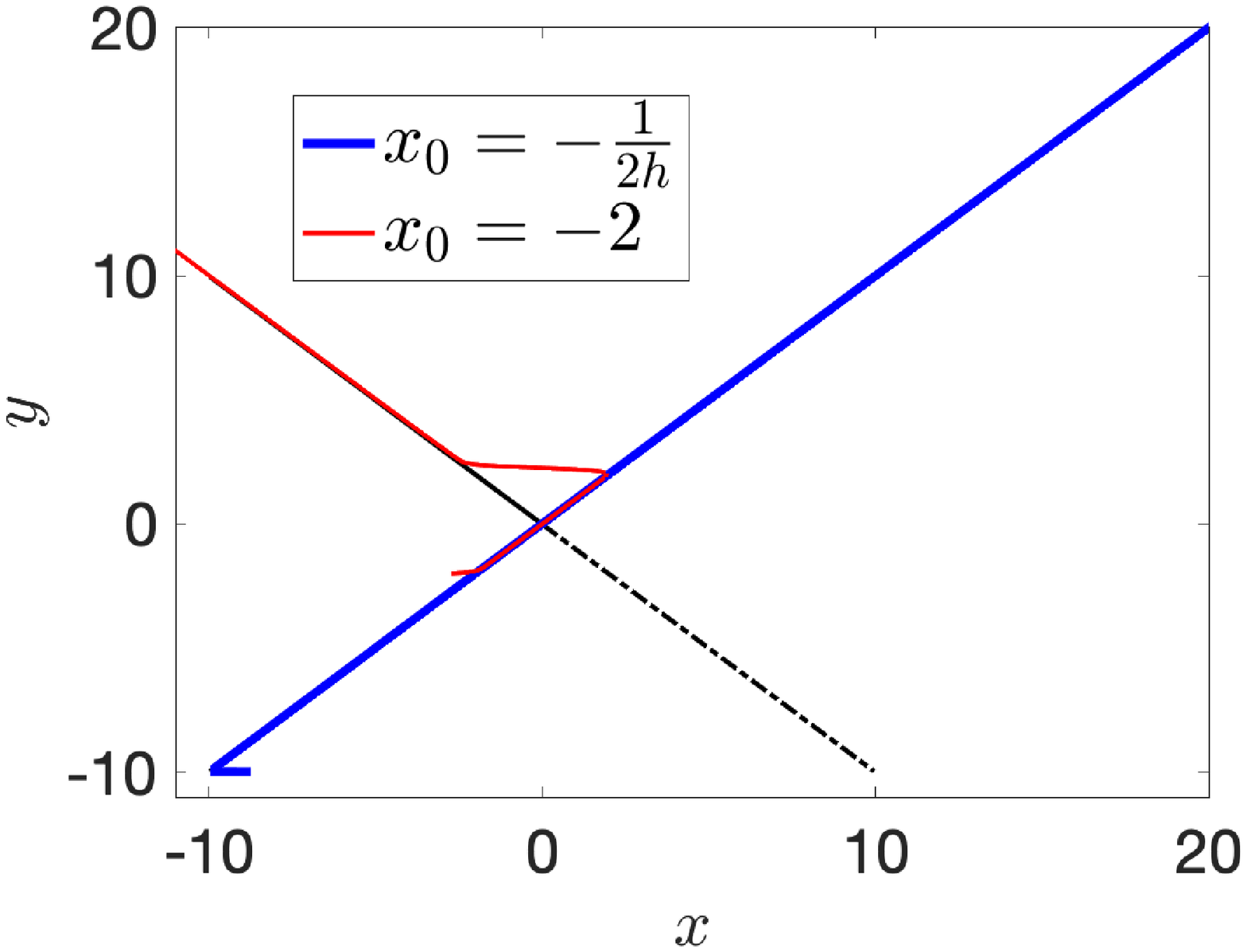}
	\caption{Numerical simulations for the transcritical equation (\ref{eq:trans}) with $\varepsilon=0.25$. We employ a time step of $h=0.1$ for the Euler discretization. We show simulations of the Euler discretization for two different initial conditions: for $x_0=-2$ (online red) which exits the manifold $y=x$ (cf. Figure~\ref{fig:trans}a) and for $x_0=-1/(2h)$ for which the modified equations predict arbitrary long stabilization (online blue).}
	\label{fig:trans_x0}
\end{figure}


\section{Discussion}
\label{sec:discuss}
The effects of a first-order scheme such as an Euler discretization can have large effects on the observed behaviour and can modify its bifurcation structure. We have shown that this can be quantitatively described by means of the modified equations which allow for a simple analytical treatment. In particular, we quantified the stabilization of canards for transcritical singularities in the continuous modified equation which closely matched those of the discrete Euler discretization of the associated dynamics. 
We suggest further investigations into similar phenomena for delayed Hopf bifurcations \cite{hayes2016geometric, neishtadt1987persistence, neishtadt1988persistence}, potentially also using modified equations for appropriate discretization schemes.

Our analysis suggests that one may modify the Euler scheme in order to obtain a better approximation of the original ODE by adding a term which would cancel the first-order correction $f_1(z)$ in the associated modified equations. For instance, the \emph{Kahan method} \cite{Kahan} 
\begin{align}
z_{n+1} &= z_n + h\left(\Id -\frac{1}{2}Df(z_n)\right)^{-1}f(z_n) \nonumber \\
&= z_n + h f(z_n) + \frac{h^2}{2}Df(z_n)f(z_n) + \mathcal{O}(h^3)
\label{eq:Kahan}
\end{align}
implies $f_1(z)=0$ (cf. Equation \ref{eq:Taylor}) and is a second order in $h$ scheme for the system $\dot z = f(z)$.
Thereby such a scheme is, of course, more accurate in terms of preserving qualitative behavior of the original ODE: it was shown in~\cite{EngelJardonSIADS, EngeletalJNLS} that the Kahan scheme (and, more generally, related A-stable methods) preserves the way-in/way-out behavior of canards and their parameter-dependent extensions.
However, for the folded canard problem, also the Kahan method does not seem to retain every property of the quadratic ODE: the first integral of motion for system~\eqref{eq:fold} is not preserved in the discretization (e.g.~\cite{EngeletalJNLS}). An extended study via a modified equation of higher order may shed some additional light on this gap.



\bibliographystyle{amsplain}

\providecommand{\bysame}{\leavevmode\hbox to3em{\hrulefill}\thinspace}
\providecommand{\MR}{\relax\ifhmode\unskip\space\fi MR }
\providecommand{\MRhref}[2]{%
  \href{http://www.ams.org/mathscinet-getitem?mr=#1}{#2}
}
\providecommand{\href}[2]{#2}

\end{document}